\documentclass[12pt]{amsart}
\usepackage{amssymb,latexsym,amsmath,amsthm,amscd}
\usepackage{setspace}
\usepackage[active]{srcltx}
\usepackage[all]{xy} \xyoption{arc}
\usepackage[left=3cm,top=2cm,right=3cm,bottom = 2cm]{geometry}
\usepackage{graphicx}
\usepackage[usenames,dvipsnames]{color}
\usepackage{charter}

\theoremstyle{plain}
\newtheorem{thm}{Theorem}[section]
\newtheorem{lem}[thm]{Lemma}
\newtheorem{prop}[thm]{Proposition}
\newtheorem{cor}[thm]{Corollary}
\newtheorem{conj}[thm]{Conjecture}

\theoremstyle{definition}
\newtheorem{dfn}[thm]{Definition}
\newtheorem{ex}[thm]{Example}

\theoremstyle{remark}
\newtheorem{rmk}[thm]{Remark}

\newcommand{\lcm}{\operatorname{lcm}}

\newcommand*{\df}{\mathrel{\vcenter{\baselineskip0.5ex \lineskiplimit0pt
                     \hbox{\scriptsize.}\hbox{\scriptsize.}}} =}

\providecommand{\abs}[1]{\left\lvert#1\right\rvert}

\providecommand{\pseries}[2]{#1[\![ #2 ]\!]}

\providecommand{\floor}[1]{{\left\lfloor #1 \right \rfloor}}

\newcommand{\QQ}{\mathbf{Q}}

\newcommand{\ZZ}{\mathbf{Z}}

\DeclareMathOperator{\denom}{denom}

\begin{document}
\title[Density formulas for bounded primes in hypergeometric series]{Density formulas for $p$-adically bounded primes for hypergeometric series with rational and quadratic irrational parameters}

\author{Cameron Franc}
\address{Department of Mathematics, McMaster University}
\curraddr{}
\email{franc@math.mcmaster.ca}
\thanks{}

\author{Nathan Heisz}
\address{Department of Mathematics, McMaster University}
\curraddr{}
\email{heiszn@mcmaster.ca}

\author{Hannah Nardone}
\address{Department of Mathematics, McMaster University}
\curraddr{}
\email{nardoneh@mcmaster.ca}

\thanks{We thank NSERC for support via a Discovery grant.}

\date{}

\begin{abstract}
We study densities of $p$-adically bounded primes for hypergeometric series in two cases: the case of generalized hypergeometric series with rational parameters, and the case of $_2F_1$ with parameters in a quadratic extension of the rational numbers. In the rational case we extend work from $_2F_1$ to $_nF_{n-1}$ for an exact formula giving the density of bounded primes for the series. The density is shown to be one exactly in accordance with the case of finite monodromy as classified by Beukers-Heckmann. In the quadratic irrational case, we obtain an unconditional lower bound on the density of bounded primes. Assuming the normality of the $p$-adic digits of quadratic irrationalities, this lower bound is shown to be an exact formula for the density of bounded primes. In the quadratic irrational case, there is a trivial upper bound of $1/2$ on the density of bounded primes. In the final section of the paper we discuss some results and computations on series that attain this bound. In particular, all such examples we have found are associated to imaginary quadratic fields, though we do not prove this is always the case.
\end{abstract}
\maketitle

\tableofcontents
\setcounter{tocdepth}{1}

\section{Introduction}

A classical problem in the study of hypergeometric series is to determine which series are integral. This problem is closely related to integrality of ratios of factorials, algebraicity of generalized hypergeometric series, and even the Riemann hypothesis --- see \cite{Bober} for a review of some of this work, and the recent paper \cite{AdolphsonSperber} for an extension of some of these questions to the case of GKZ-hypergeometric series. The related problem of determining when a generalized hypergeometric series has finite monodromy was solved by Beukers-Heckmann \cite{beukers}, preceded in the case of $_2F_1$ by classical work of Schwarz. In these cases the algebraic series have rational hypergeometric parameters, and it is known by a classical result of Eisenstein that while the series need not have integral coefficients, there are at worst a finite number of primes that arise in the denominators of the series. Thus, while algebraic generalized hypergeometric series need not be integral, the density of primes where such a series is $p$-adically bounded equals one.

In this paper we study densities of $p$-adically bounded primes for various hypergeometric series, typically non-algebraic. The study of \(p\)-adic boundedness of a hypergeometric series, spearheaded by Dwork \cite{dwork} and Christol \cite{christol}, provided necessary and sufficient conditions for a generalized hypergeometric series \(_nF_{n-1}\) to be \(p\)-adically bounded. Franc-Gannon-Mason \cite{FGM} showed that for all series $_2F_1$ with rational hypergeometric parameters, there is a corresponding density of $p$-adically bounded primes, and this density is one exactly when the series is algebraic, that is, it comes from Schwarz's classical list. Franc et.al. in \cite{fetal} built on this work to establish an exact formula for the Dirichlet densities of bounded primes in the denominators of series \(_2F_1\) with rational parameters. Our first main result Theorem \ref{t:rational} extends this density formula to generalized hypergeometric series \(_mF_n\) and furthermore shows that the only case of interest in terms of boundedness of primes is the case where \(m = n+1\). These results can be summarized as follows:
\begin{thm}
    \label{t:main1}
Let $F(z) = {}
_mF_n(\alpha_1,\ldots, \alpha_m; \beta_1,\ldots, \beta_n;z)$ denote a generalized hypergeometric series with rational hypergeometric parameters, $\alpha_i,\beta_j \in \QQ$. Then the set $B$ of primes where $F(z)$ is $p$-adically bounded has a Dirichlet density. This density is one if $m > n+1$, and it is zero of $m < n+1$. When $m=n+1$, the density can be computed exactly using a simple formula given in Corollary \ref{c:rationaldensityformula}.  In this case, the density is one if and only if $F(z)$ is an algebraic function.
\end{thm}

The next case we discuss, starting in Section \ref{s:quadratic}, is that of hypergeometric series of the form
\[
f(z) = {}_2F_1(a+b\sqrt{D},a-b\sqrt{D};c;z)
\]
where $a,b,c,D \in \QQ$. Such series have rational expansions yet, if $D$ is not square, then they are transcendental functions by \cite{beukers}. The \(N\)-integrality of hypergeometric series \(_mF_n\) with parameters from quadratic fields has been studied by Hong-Wang \cite{hong}, and specific cases of \(_2F_1\) with quadratic irrational parameters have appeared as solutions to a differential equation concerning the arithmetic of modular forms on \(\Gamma_0(2)\) in work by Gottesman \cite{gottesman}.  Some of our results in Section \ref{s:quadratic} can be summarized as follows:
\begin{thm}
    \label{t:main2}
Let $f(z) \in \pseries{\QQ}{z}$ be a hypergeometric series as above, where $D$ is not a perfect square, and let $B$ be the set of primes where $f(z)$ is $p$-adically bounded. Then $B$ has a Dirichlet density that is at most $1/2$, and which is bounded below by explicit formulas described in Propositions \ref{prop1} and \ref{prop2}. If one assumes that the $p$-adic digits of quadratic irrationalities are normal, then the density bounds in Propositions \ref{prop1} and \ref{prop2} are exact formulas.
\end{thm}
In the rational case of Theorem \ref{t:main1}, our arguments rely on the periodic $p$-adic expansions of rational numbers. In the quadratic irrational case of Theorem \ref{t:main2}, the expansions are no longer periodic, and so normality serves as a convenient replacement for this property. Actually, to establish our conjectural density formulas in Section \ref{s:quadratic}, we only need to assume a result about $p$-adic expansions of quadratic irrationalities, Conjecture \ref{conj} below, that is implied by $p$-adic normality of digits. We do not know if Conjecture \ref{conj} is in fact equivalent to $p$-adic normality. 

The density in Theorem \ref{t:main2} is bounded above by $1/2$ for the simple reason that a series such as $f(z)$ can only be $p$-adically bounded for primes $p$ that split in $\QQ(\sqrt{D}$), cf. Proposition \ref{p:split}. The reason why there are two density formulas, cf. Propositions \ref{prop1} and \ref{prop2}, is because the formulas are different in the cases where $2a$ is integral or not. Since the trace of $a+b\sqrt{D}$ plays an important role in the proofs, whether the trace $2a$ has a truncating or periodic $p$-adic expansion affects the argument and the formulas.

Section \ref{s:schwarz} is inspired by the Schwarz list for ${}_2F_1$, which essentially classifies series with density of bounded primes equal to one. For the case of series such as $f(z)$ above, a natural quadratic irrational analogue is to classify series where the density of bounded primes is $1/2$. We do not resolve this question, but we discuss some computations towards it, and also establish some upper-bounds that indicate that it should be reasonably rare for a series such as $f(z)$ to have a density of bounded primes equal to $1/2$. For example, in all of our examples, we only ever observed a density of $1/2$ when the field $K=\QQ(\sqrt{D})$ is imaginary quadratic. See Table \ref{tab:dvaryk} for some examples with $a$ and $c$ of small height. This rarity of large densities is in-line with the classical Schwarz list being rather sparse among all rational hypergeometric series.

An interesting open question is how the irrational case extends to higher values of \(n\) for a general series \(_nF_{n-1}\) with irrational but algebraic hypergeometric parameters. For example, one could consider series \(_3F_2\) with a pair of quadratic conjugates in the numerator, along with a third rational numerator parameter, and a separate pair of conjugate quadratic irrationalities in the denominator. This will still have rational coefficients, and presumably also a density of bounded primes. Or, one could consider $_3F_2$ with three cyclic Galois conjugates of degree three as the numerator parameters --- again, is there a density of bounded primes for such series? While we find these questions interesting, the array of conceivable patterns only grows more vast as \(n\) increases, so it is difficult to imagine what sort of general result might exist. Nevertheless, the density results in the quadratic irrational case of ${}_2F_1$ are surprisingly clean, and they require a nice application of $p$-adic normality of algebraic numbers --- or its a priori weaker form of Conjecture \ref{conj} ---  that could provide an impetus for work in the area of $p$-adic metric number theory. 

\section{Background}
The rising factorials are defined for integers $k\geq 0$ as $(x)_0 = 1$ and otherwise
\[
  (x)_k \df x(x+1)\cdots (x+k-1).
\]
Let $\alpha = (\alpha_1,\ldots ,\alpha_{m})$ and $\beta = (\beta_1,\ldots, \beta_m)$ denote complex numbers, and assume that $\beta_j \not \in \ZZ_{\leq 0}$ for all $j=1,\ldots, n$. The corresponding \emph{generalized hypergeometric series}, or just hypergeometric series, ${}_mF_n(\alpha;\beta;z)$ is defined by the infinite series
\[
{}_mF_n(\alpha;\beta;z) = \sum_{k\geq 0} \frac{(\alpha_1)_k\cdots (\alpha_m)_k}{(\beta_1)_k\cdots (\beta_n)_k} \frac{z^k}{k!}.
\]
These series satisfy hypergeometric differential equations, and the equations are regular singular exactly when $m=n+1$. The case of ${}_2F_1$ has a long history dating back at least to work of Gauss and Riemann. The finite monodromy groups of these equations were classified by Schwarz, and the case of ${}_mF_n$ was treated more recently by Beukers-Heckmann\cite{beukers}.

Let $p$ be a rational prime, and let $\nu_p$ denote the corresponding $p$-adic valuation normalized so that $\nu_p(p)=1$.
\begin{dfn}
  A hypergeometric series for parameters $(\alpha,\beta)$ is said to be \emph{rational} if ${}_mF_n(\alpha;\beta;z) \in \pseries{\QQ}{z}$.
\end{dfn}

\begin{rmk}
  Obviously, hypergeometric parameters such that $\alpha_i\in \QQ$ and $\beta_j \in \QQ$ for every $i$ and $j$ yield rational series, but there exist other examples such as $\alpha = (a+b\sqrt{2},a-b\sqrt{2})$ and $\beta = (c)$ for any rational numbers $a,b,c$ with $c \not \in \ZZ_{\leq 0}$.
\end{rmk}
\begin{dfn}
  A prime $p$ is said to be \emph{bounded} for a rational hypergeometric series if the $p$-adic valuations of the coefficients of ${}_mF_n(\alpha;\beta;z)$ are bounded below. Equivalently, there exists $N \in \ZZ$ such that
  \[
    p^N{}_mF_n(\alpha;\beta;z) \in \pseries{\ZZ_p}{z}.
  \]
  Otherwise $p$ is said to be \emph{unbounded} for the given series.
\end{dfn}

In \cite{FGM} it was shown that if $\alpha,\beta,\gamma \in \QQ$ with $\gamma \not \in \ZZ_{\leq 0}$, then the set of primes such that the classical hypergeometric series ${}_2F_1(\alpha,\beta;\gamma;z)$ has bounded denominators has a natural density. In Theorem 4 of \cite{fetal}, precision was added to this result in the form of an explicitly computable formula for this natural density:
\begin{thm}\label{t:fetal}
  Let $a,b$ and $c$ denote rational numbers with $0 < a,b,c, < 1$, where $c \neq a,b$. Let $N$ denote the least common multiple of the denominators of $a-1$, $b-1$ and $c-1$, and define
  \[
  B(a,b;c) \df \{u \in (\ZZ/N\ZZ)^\times \mid \textrm{ for all } j \in \ZZ,~ \{-u^jc\} \leq \max(\{-u^ja\},\{-u^jb\})\}.
\]
Then for all primes $p > N$, the series ${}_2F_1(a,b;c;z)$ is $p$-adically bounded if and only if $p$ is congruent to an element of $B(a,b;c)$ mod $N$. Therefore, the set of bounded primes for ${}_2F_1(a,b;c;z)$ has a density equal to $\tfrac{\abs{B(a,b;c)}}{\phi(N)}$.
\end{thm}
\begin{proof}
    This is Theorem 4 of \cite{fetal}.
\end{proof}

Our goals below are twofold:
\begin{enumerate}
\item generalize Theorem \ref{t:fetal} to the case of general $m$ and $n$;
\item generalize Theorem \ref{t:fetal} when $m=2$, $n=1$, to the case of certain rational parameters living in quadratic number fields.
\end{enumerate}
In the case of quadratic fields, our density result is conditional on a seemingly plausible conjecture about the normality of the $p$-adic expansions of irrational algebraic numbers at split primes. Such a conjecture is a substitute for the lack of periodicity of the $p$-adic expansions of the hypergeometric parameters in these cases.

To study the $p$-adic boundedness of hypergeometric series, following \cite{FGM} we introduce:
\begin{dfn}
  \label{d:carries}
  Let $c_p$ denote the function
  \[c_p \colon \ZZ_p\times \ZZ_{\geq 0} \to \ZZ_{\geq 0} \cup \{\infty\}\] where $c_p(x,k)$ computes the number of carries needed to evaluate the $p$-adic sum $x+k$ using the usual add-and-carry method.
\end{dfn}
In \cite{FGM} it was observed that this function is locally constant in $x$, and using a classical result of Kummer, one in fact has the following:
\begin{thm}[Kummer]
  \label{t:kummer}
One has $\nu_p\binom{x+k}{k} = c_p(x,k)$.
\end{thm}
\begin{proof}
See Theorem 2.5 of \cite{FGM}.
\end{proof}

The following result is a slight generalization of Theorem 3.4 of \cite{FGM}, to the case of general $m$ and $n$.
\begin{cor}
  \label{c:FGM}
  Let $\alpha$ and $\beta$ denote hypergeometric parameters contained in $\ZZ_p$, such that no $\beta$ parameter is a negative integer. Then if we write ${}_m F_n(\alpha;\beta;z) = \sum_{k\geq 0}A_kz^k$ for $A_k \in \QQ_p$, we have
  \[
  \nu_p(A_k) = \sum_{i=1}^mc_p(\alpha_i-1,k)-\sum_{j=1}^n c_p(\beta_j-1,k) + (m-n-1)\nu_p(k!).
  \]
\end{cor}
\begin{proof}
  By a direct computation one sees that
  \[
  A_k=\frac{\prod_{i=1}^m\binom{\alpha_i-1+k}{k}}{\prod_{j=1}^n\binom{\beta_j-1+k}{k}}\cdot (k!)^{m-n-1}
\]
Therefore, this Corollary follows immediately from Theorem \ref{t:kummer}.
\end{proof}

\begin{rmk}
  \label{r:nontrivial}
Since $\nu_p(k!) = O(k)$, while $c_p(x,k) = O(\log k)$ away from the poles $x$ of $c_p(x,k)$, one can use Corollary \ref{c:FGM} to prove that for rational hypergeometric parameters, all primes are bounded if $m > n+1$, whereas no primes are bounded if $m < n+1$. Thus, as far as densities go, for rational parameters the only interesting case is when $m=n+1$ and the corresponding hypergeometric differential equation is regular singular.
\end{rmk}
\section{Rational parameters}
\label{s:rational}

This section considers the case of rational parameters $\alpha,\beta$. In light of Remark \ref{r:nontrivial}, we restrict also to the case of ${}_nF_{n-1}$. We begin by recalling some facts on the $p$-adic expansions of rational numbers, which are always eventually periodic. In fact, if $a$ is a rational number satisfying $-1 < a < 0$ and $a \in \ZZ_p^\times$, then the $p$-adic expansion of $a$ is purely periodic, of period equal to the multiplicative order of $p$ modulo the denominator of $a$. Since the density of bounded primes for a given set of hypergeometric parameters only depends on the rational parameters mod $\ZZ$, we will normalize our parameters to lie in the interval $(0,1)$. We state the following lemma on the digits of our hypergeometric parameters in a way which will make the application of Corollary \ref{c:FGM} more straightforward.

\begin{lem}
  \label{l:digitformula}
  Let $a = \tfrac uv$ denote a rational number with $\gcd(u,v) =1$ satisfying $0 < a < 1$, and let $p$ denote a prime such that $a-1$ is a $p$-adic unit. Let $M$ denote the multiplicative order of $p \pmod{v}$, and let the periodic expansion of $a-1$ be denoted as:
  \[
  a-1 = \overline{\alpha_0\alpha_1\cdots \alpha_{M-1}}.
\]
Then for each $j$ with $0\leq j \leq M-1$ we have
\[
  \alpha_j = \floor{\{-p^{M-1-j}a\}p}.
\]
In particular, if $p > v$, then each $\alpha_j$ is nonzero.
\end{lem}
\begin{proof}
Most of this is Lemma 2.4 of \cite{FGM}. All that remains to observe is the final claim that each $\alpha_j$ is nonzero. For this, begin with the observation that
\[
\{-p^{M-1-j}a\} -\tfrac 1p < \tfrac{\alpha_j}{p} \leq \{-p^{M-1-j}a\}.
\]
Since $p$ is coprime to $a$, then $\{-p^{M-1-j}a\}$ is a rational number of the form $k/v$ for some integer $1\leq k \leq v-1$. Hence,
\[
0<\frac{1}{v}-\frac 1p < \frac{\alpha_j}{p}
\]
as claimed.
\end{proof}

\begin{dfn}
A prime $p$ is \emph{good} for hypergeometric parameters $\alpha$, $\beta$ provided that all of $\alpha_i-1$ and $\beta_i-1$ are $p$-adic units, and such that $p$ does not divide any of the differences $\alpha_i-\beta_j$.
\end{dfn}

Note that as long as no $\alpha_i$ is equal to a $\beta_j$, a harmless hypothesis that we enforce below, then the set of good primes for a set of hypergeometric parameters contains all primes that are large enough. Hence, as far as densities of bounded primes go, it is harmless to restrict to considering good primes only.

Let $p$ denote a good prime for some rational hypergeometric parameters, and let $M$ denote the order of $p$ modulo the least common multiple of the denominators of the hypergeometric parameters. Then by Lemma \ref{l:digitformula}, the $p$-adic periods of each parameter all divide $M$.

\begin{dfn}
  \label{d:numeratormajorized}
Let $\alpha = (\alpha_i)_{i=1}^n$ and $\beta = (\beta_i)_{i=1}^{n-1}$ denote rational hypergeometric parameters, and let $N$ denote the least common multiple of the denominators of the parameters. Fix an invertible congruence class $u\pmod{N}$. Then a set of hypergeometric parameters is said to be \emph{numerator majorized} for this congruence class provided that for all integers $j\geq 0$, there exists a permutation $\sigma_j \in S_n$ such that
  \[
  \{-u^j\beta_i\} \leq \{-u^j\alpha_{\sigma_j(i)}\}
\]
for all $i=1,\ldots, n-1$.
\end{dfn}
Definition \ref{d:numeratormajorized} is reminiscient of the interlacing of roots of unity condition that appears in the study of finite monodromy in \cite{Landau} and \cite{beukers}. For a given set of parameters, it is a finite computation to determine which congruence classes are numerator majorized or not.
\begin{ex}
  \label{ex1}
  Consider the hypergeometric series
  \[
  {}_3F_2(\tfrac 15, \tfrac 25, \tfrac 35; \tfrac 45, \tfrac 12;z).
\]
In this case $N=10$, and so we only need to test the numerator majorization condition for $j=0,1,2,3$, since $\phi(10) = 4$, and thus $p^4 \equiv 1 \pmod{10}$.
\begin{table}
  \renewcommand{\arraystretch}{1.5}
\begin{tabular}{c||c|c|c||c|c}
  $p\pmod{10}$ & $\tfrac 15$ & $\tfrac 25$ & $\tfrac 35$ & $\tfrac 45$ & $\tfrac 12$\\
  \hline
  $1$ &$\left[\frac{4}{5}, \frac{4}{5}, \frac{4}{5}, \frac{4}{5}\right]$&$\left[\frac{3}{5}, \frac{3}{5}, \frac{3}{5}, \frac{3}{5}\right]$&$\left[\frac{2}{5}, \frac{2}{5}, \frac{2}{5}, \frac{2}{5}\right]$&$\left[\frac{1}{5}, \frac{1}{5}, \frac{1}{5}, \frac{1}{5}\right]$&$\left[\frac{1}{2}, \frac{1}{2}, \frac{1}{2}, \frac{1}{2}\right]$\\
  $3$ &$\left[\frac{4}{5}, \frac{2}{5}, \frac{1}{5}, \frac{3}{5}\right]$&$\left[\frac{3}{5}, \frac{4}{5}, \frac{2}{5}, \frac{1}{5}\right]$&$\left[\frac{2}{5}, \frac{1}{5}, \frac{3}{5}, \frac{4}{5}\right]$&$\left[\frac{1}{5}, \frac{3}{5}, \frac{4}{5}, \frac{2}{5}\right]$&$\left[\frac{1}{2}, \frac{1}{2}, \frac{1}{2}, \frac{1}{2}\right]$\\
  $7$ &$\left[\frac{4}{5}, \frac{3}{5}, \frac{1}{5}, \frac{2}{5}\right]$&$\left[\frac{3}{5}, \frac{1}{5}, \frac{2}{5}, \frac{4}{5}\right]$&$\left[\frac{2}{5}, \frac{4}{5}, \frac{3}{5}, \frac{1}{5}\right]$&$\left[\frac{1}{5}, \frac{2}{5}, \frac{4}{5}, \frac{3}{5}\right]$&$\left[\frac{1}{2}, \frac{1}{2}, \frac{1}{2}, \frac{1}{2}\right]$\\
  $9$ &$\left[\frac{4}{5}, \frac{1}{5}, \frac{4}{5}, \frac{1}{5}\right]$&$\left[\frac{3}{5}, \frac{2}{5}, \frac{3}{5}, \frac{2}{5}\right]$&$\left[\frac{2}{5}, \frac{3}{5}, \frac{2}{5}, \frac{3}{5}\right]$&$\left[\frac{1}{5}, \frac{4}{5}, \frac{1}{5}, \frac{4}{5}\right]$&$\left[\frac{1}{2}, \frac{1}{2}, \frac{1}{2}, \frac{1}{2}\right]$
\end{tabular}
\caption{The values $\{-p^j\alpha_i\}$ and $\{-p^j\beta_i\}$ for the hypergeometric parameters $\alpha = (\tfrac 15,\tfrac 25, \tfrac 35)$ and $\beta = (\tfrac 45, \tfrac 12)$.}
\label{t1}
\end{table}
Table \ref{t1} shows that the class $1 \pmod{10}$ is numerator majorized for these parameters. The class $3\pmod{10}$ fails the numerator majorization condition for $j=2$, as the third entry of the list for the denominator parameter $\tfrac 45$ in that case is $\tfrac 45$, which is maximal and thus can't be numerator majorized. Similarly, $7 \pmod{10}$ and $9\pmod{10}$ fail the numerator majorization condition in this example.
\end{ex}

\begin{lem}
  \label{l:equivalent}
  Let $\alpha = (\alpha_i)_{i=1}^n$ and $\beta = (\beta_i)_{i=1}^{n-1}$ denote rational hypergeometric parameters taken from the interval $(0,1)$, and let $N$ denote the least common multiple of the denominators of the parameters. Then these parameters are numerator majorized for a a congruence class $u\pmod{N}$ if and only if for all large enough primes $p\equiv u \pmod{N}$, if we write the $p$-adic digits of $\alpha_i-1$ as $\alpha_{i,j}(p)$ and similarly for $\beta_i-1$ and $\beta_{i,j}(p)$, then for all integers $j\geq 0$ there exists a permutation $\sigma_j \in S_n$ such that
  \[
  \beta_{i,j}(p) \leq \alpha_{\sigma_j(i),j}(p)
\]
for all $i=1,\ldots, n-1$.
\end{lem}
\begin{proof}
  For all primes $p\equiv u \pmod{N}$, being numerator majorized is equivalent to the condition
  \[
  \{-p^j\beta_i\}p \leq \{-p^j\alpha_{\sigma_j(i)}\}p.
\]
If $p > 1/(\alpha_u-\beta_v)$ for all $u,v$, then this inequality is \emph{equivalent} to the inequality obtained by taking floors above:
  \[
  \floor{\{-p^j\beta_i\}p} \leq \floor{\{-p^j\alpha_{\sigma_j(i)}\}p}.
\]
Therefore, the equivalence of these two formulations of numerator majorization follows from Lemma \ref{l:digitformula}.
\end{proof}

\begin{thm}
  \label{t:rational}
Let $\alpha = (\alpha_i)_{i=1}^n$ and $\beta = (\beta_i)_{i=1}^{n-1}$ denote hypergeometric parameters, and let $N$ denote the least common multiple of the denominators of the parameters, and fix an invertible congruence class $u\pmod{N}$. Then the corresponding hypergeometric series ${}_nF_{n-1}(\alpha,\beta,z)$ is $p$-adically bounded for all sufficiently large primes $p\equiv u \pmod{N}$ if and only if the hypergeometric parameters are numerator majorized with respect to the congruence class $p \equiv u\pmod{N}$.
\end{thm}
\begin{proof}
  Assume that $p > \max_{i,j}(\alpha_i-\beta_j)^{-1}$ so that Lemma \ref{l:equivalent} holds for such primes. Assume that the series is numerator majorized with respect to $u \pmod{N}$, so that the $p$-adic digits are numerator majorized by Lemma \ref{l:equivalent}. Now Corollary \ref{c:FGM} shows that the negative contributions for carries from terms $c_p(\beta_j-1,k)$ are compensated for by carries $c_p(\alpha_i-1,k)$. Hence $\nu_p(a_k) \geq 0$ for such primes $p$, so that in fact ${}_nF_{n-1}(\alpha,\beta,z) \in \pseries{\ZZ_p}{z}$. This proves one direction of the Theorem.

  Assume conversely that the parameters are not numerator majorized with respect to a congruence class $u\pmod{N}$. This means there exists a digit index $j$ such that for every $\sigma \in S_n$, there exists a hypergeometric parameter index $i$ such that
  \[
  \beta_{i,j}(p) > \alpha_{\sigma(i),j}(p).
\]
First suppose that there is an index $i$ such that $\beta_{i,j}(p) > \alpha_{k,j}(p)$ for all $k$. For each integer $A \geq 0$, consider the term of the hypergeometric series indexed by
\[k_A \df (p-\beta_{i,j}(p))\sum_{u=0}^Ap^{j+u\phi(N)}.\]
Since $\phi(N)$ is a common period for the $p$-adic expansions of all of the hypergeometric parameters, we see that $c_p(\beta_i-1,k_A) = A+1$, but otherwise by construction $c_p(\beta_u-1,k_A) = 0$ for $u\neq i$ and $c_p(\alpha_v-1,k_A) = 0$ for all $v$. Therefore Corollary \ref{c:FGM} shows that in this case $v_p(a_{k_A}) \geq A+1$, where the $a_{k}$ denote the coefficients of ${}_nF_{n-1}(\alpha,\beta,z)$. This treats the case where one of the $\beta_{i,j}(p)$ digits is maximal among all of the $j$-th coefficients of the given set of hypergeometric parameters.

Otherwise, we can without loss of generality assume that $\alpha_{1,j}(p) \geq \cdots \geq \alpha_{n,j}(p)$ and similarly $\beta_{1,j}(p) \geq \cdots \geq \beta_{n-1,j}(p)$. We have treated in the previous paragraph the case when $\alpha_{1,j}(p) < \beta_{1,j}(p)$, and since numerator majorization fails, there must exist a smallest index $i$ such that $\alpha_{u,j}(p) \geq \beta_{u,j}(p)$ for $1\leq u \leq i-1$ but $\beta_{i,j}(p) > \alpha_{i,j}(p)$. Now defining $k_A$ as above, this time we see that carries coming from the $\beta_{u,j}(p)$ are compensated for by carries coming from $\alpha_{u,j}(p)$ for $1 \leq u \leq i-1$, but we have $c_p(\alpha_{i,j}(p),k_A) \geq A+1$ and $c_p(\beta_{i,j}(p),k_A) =0$. Therefore we again find that $\nu_p(a_{k_A}) \leq -A-1$, and so in all cases, if the parameters are not numerator majorized for $u \pmod{N}$, then the corresponding series is not $p$-adically bounded for all primes $p > \max_{i,j}(\alpha_i-\beta_j)^{-1}$ satisfying $p\equiv u\pmod{N}$. This concludes the proof of the theorem.
\end{proof}

\begin{rmk}
\label{r:explicit}
  We emphasize that this proof shows that if the parameters are numerator majorized for $u\pmod{N}$, then for all primes $p > \max_{i,j}(\alpha_i-\beta_j)^{-1}$ satisfying $p\equiv u \pmod{N}$, we have ${}_nF_{n-1}(\alpha,\beta,z) \in \pseries{\ZZ_p}{z}$.
\end{rmk}

\begin{cor}
  \label{c:rationaldensityformula}
  Let $\alpha = (\alpha_i)_{i=1}^n$ and $\beta = (\beta_i)_{i=1}^{n-1}$ denote hypergeometric parameters, and let $N$ denote the least common multiple of the denominators of the parameters. Then for every prime $p > \max_{i,j}(\alpha_i-\beta_j)^{-1}$, the series ${}_nF_{n-1}(\alpha,\beta,z)$ is $p$-adically bounded if and only if $p$ is congruent to an element of the following set:
  \[
  B(\alpha;\beta) =\{u \in (\ZZ/N\ZZ)^\times \mid \alpha \textrm{ and } \beta \textrm{ are numerator majorized for } u\}.
\]
In fact, for such primes $p$, the series ${}_nF_{n-1}(\alpha,\beta,z)$ is $p$-adically integral. In particular, the set of bounded primes for $\alpha,\beta$ has a Dirichlet density equal to $\tfrac{\abs{B(\alpha;\beta)}}{\phi(N)}$.
\end{cor}
\begin{proof}
This follows directly from Theorem \ref{t:rational} and Remark \ref{r:explicit}.
\end{proof}

\begin{rmk}
\label{r:cyclic}
  Notice that the set $B(\alpha;\beta)$, if nonempty, is a union of cyclic subgroups of $(\ZZ/N\ZZ)^\times$.
\end{rmk}

\begin{ex}
  \label{ex2}
In Example \ref{ex1} we saw that for
  \[
  {}_3F_2(\tfrac 15, \tfrac 25, \tfrac 35; \tfrac 45, \tfrac 12;z)
\]
we have $N = 10$ and $B(\tfrac 15, \tfrac 25, \tfrac 35; \tfrac 45, \tfrac 12) = \{1\} \subseteq (\ZZ/10\ZZ)^\times$. Therefore, the bounded primes $p > 10$ for this series are precisely those satisfying $p\equiv 1 \pmod{10}$, so that the  density of bounded primes in this case is $\tfrac 14$. In fact, this series is $p$-integral for such primes, so that the only primes appearing in the denominators of this series satisfy $p \leq 7$ or $p\equiv 3,7,9 \pmod{10}$.
\end{ex}

The following result generalizes Theorem 4.14 from \cite{FGM}.
\begin{cor}
  Let $\alpha = (\alpha_i)_{i=1}^n$ and $\beta = (\beta_i)_{i=1}^{n-1}$ denote hypergeometric parameters, and assume that they are ordered in increasing order of their fractional parts. Then the density of bounded primes for ${}_nF_{n-1}(\alpha,\beta,z)$ is zero if and only if there exists an index $i$ such that
  \[
  \{\beta_i\} < \{\alpha_i\}.
  \]
\end{cor}
\begin{proof}
By Corollary \ref{c:rationaldensityformula} and Remark \ref{r:cyclic}, the density of bounded primes is zero if and only if $1 \not \in B(\alpha,\beta)$. The numerator majorization failure for $u=1$ is equivalent to the existence of an index $i$ with $\{-\beta_i\} > \{-\alpha_i\}$, or equivalently, $\{\beta_i\} < \{\alpha_i\}$ as claimed.
\end{proof}

\section{Quadratic irrational parameters}
\label{s:quadratic}
Let now $D \in \QQ$ be a nonsquare rational, and let $K = \QQ(\sqrt{D})$. Notice that for $a,b,c\in\QQ$ with $c \not \in \ZZ_{\leq 0}$, the series
\[
  F_D(a,b,c,z) \df {}_2F_1(a+b\sqrt{D},a-b\sqrt{D};c;z)
\]
has rational coefficients. We will assume $b \neq 0$, so that $F_D$ does not have rational parameters.
\begin{prop}
\label{p:split}
For all but finitely many primes $p$, a necessary condition for $F_D(a,b,c,z)$ to be $p$-adically bounded is that $p$ splits in $K$.
\end{prop}
\begin{proof}
  Let $A_n$ denote the $n$th coefficient of $F_D(a,b,c,z)$, and let $P(x)$ denote the minimal polynomial of $a+b\sqrt{D}$:
  \[P(x) = x^2 - 2ax + a^2-b^2D.\]
  Then we find that
  \[
    A_n = \frac{(a+b\sqrt{D})_n(a-b\sqrt{D})_n}{(c)_nn!} =\frac{\prod_{j=0}^{n-1}P(-j)}{(c)_nn!}.
  \]
  Since $a+b\sqrt{D}$ generates $K/\QQ$, and $P(x)$ is its minimal polynomial, with at most finitely many exceptions only rational primes that split in $K$ divide the values $P(j)$ in the numerator of $A_n$. On the other hand, the denominator is $(c)_nn!$, and all but finitely many exceptions depending only on $c$ divide this denominator to larger and larger powers as $n$ grows. Therefore, a necessary condition to have cancellation in $A_n$ for large enough primes is that $p$ be split in $K$.
\end{proof}

\begin{rmk}
We  say \(u\) splits in \(K\) if primes equivalent to \(u\mod M\) split in \(K\). This is well defined as the splitting condition of a prime in a number field \(K\) relies on the congruence class of the prime \(\mod \Delta_K\), where \(\Delta_K\) denotes the discriminant of the number field \(K\). Since \(M = \lcm(\denom(a),\denom(c),\Delta_K)\), all primes equivalent to \(u\mod M\) split in \(K\) if any prime equivalent to \(u\mod M\) splits in \(K\).
\end{rmk}

Let $S$ denote the set of rational primes that split in $K$. Notice that Corollary \ref{c:FGM} still applies to $F_D(a,b,c,z)$ for all but finitely many primes $p \in S$. In the rational case we are able to leverage the periodicity of $p$-adic expansions in conjunction with Corollary \ref{c:FGM} to produce unbounded denominators. But the $p$-adic expansions of the irrational hypergeometric parameters $\alpha = (a+b\sqrt{D},a-b\sqrt{D})$ are no longer periodic. However, it turns out that we can still study the density of bounded primes in a similar way if these $p$-adic irrational numbers have sufficiently randomly distributed digits. To begin, we recall the following definition.

\begin{dfn}
\label{d:normal}
  A $p$-adic integer $\alpha \in \ZZ_p$ is said to be \emph{normal} provided that every sequence of $p$-adic digits occurs equally often in the following asymptotic sense: let $\alpha = \sum_{n\geq 0}\alpha_np^n$ be the $p$-adic expansion of $\alpha$ for $\alpha_n \in \{0,1,\ldots, p-1\}$ for all $n$, and let $B = B_1B_2\cdots B_{m}$ be a bit string, where $B_j \in \{0,1,\ldots, p-1\}$ for all $j$. Then
  \[
  \lim_{N\to \infty}\frac{\abs{\{j\leq N \mid \alpha_j\alpha_{j+1}\cdots \alpha_{j+m-1} = B\}}}{N} = \frac{1}{p^m}.
\]
This should hold for any choice of $p$-adic bit string $B$.
\end{dfn}

Normality is an established property over the real numbers. For more information see a text on metric number theory, such as \cite{MNT}. 

There exist $p$-adic transcendental numbers that are not normal. Computations suggest that $p$-adic numbers that are algebraic over $\QQ$ of degree $d\geq 2$ may always be normal, though this is currently an open question. The following conjecture may be regarded as a weakened form of normality.

\begin{conj}
  \label{conj}
  Let $\alpha \in \overline \QQ\setminus \QQ$, and let $S$ be the set of rational primes that are totally split in $\QQ(\alpha)$. Let $r$ and $s$ be integers, and let $u,v$ be real numbers with $0\leq u < v \leq 1$. Let $\alpha_n(p,\sigma)$ denote the $n$th $p$-adic digit of $\alpha$ for $p \in S$ with respect to the field embedding $\sigma \colon \QQ(\alpha) \to \QQ_p$. Then for all but finitely many primes $p \in S$ (where the finite exceptional set depends on $\alpha$, $r$, $s$, $u$ and $v$),  the digit $\alpha_n(p,\sigma)$ is contained in the interval $(u(p-1),v(p-1))$ for infinitely many integers $n\geq 0$ in the arithmetic progression $n = rm+s$.
\end{conj}

\begin{prop}
  \label{p:normalimpliesconj}
Suppose that $\alpha \in \overline\QQ \setminus \QQ$ has the property that $\sigma(\alpha)$ is $p$-adically normal for any embedding $\sigma \colon \QQ(\alpha) \to \QQ_p$. Then Conjecture \ref{conj} holds for $\alpha$.
\end{prop}
\begin{proof}
We will show that given an arithmetic progression $n = rm+s$, then \emph{any} digit occurs infinitely often in the sequence $(\alpha_{rm+s}(p,\sigma))_{m\geq 0}$ of $p$-adic digits, which clearly implies Conjecture \ref{conj}.

Suppose that a digit $b$ only occurs finitely many times along this progression, and let $B$ be the string which is $m$ copies of $b$. Then the quantity
\[
\abs{\{j\leq N \mid \alpha_j(p,\sigma)\alpha_{j+1}(p,\sigma)\cdots \alpha_{j+m-1}(p,\sigma) = B\}}
\]
is bounded absolutely, independently of $N$, since at least one of the terms $\alpha_{j+i}(p,\sigma)$ for $0\leq i \leq m-1$ must meet the sequence $(\alpha_{rm+s}(p,\sigma))_{m\geq 0}$. As this contradicts the normality of $\sigma(\alpha)$, it thus verifies our claim that every digit appears infinitely often in $(\alpha_{rm+s}(p,\sigma))_{m\geq 0}$. This concludes the proof.
\end{proof}

Before establishing our density results, we reformulate the condition of $p$-adic unboundedness for $F_D$ slightly:
\begin{thm}
  Let $a,b,c$ denote rational numbers as above, and write the $p$-adic digits of $a+b\sqrt{D}-1$, $a-b\sqrt{D}-1$, and $c-1$ as $\alpha_j(p)$, $\alpha_j'(p)$ and $\gamma_j(p)$, respectively. Then the following are equivalent for all but finitely many primes that split in $K$:
  \begin{enumerate}
  \item the inequality $\gamma_j(p) > \max(\alpha_j(p),\alpha'_j(p))$ holds for infinitely many $j$;
  \item $F_D(a,b;c;z)$ has $p$-adically unbounded coefficients.
  \end{enumerate}
\end{thm}
\begin{proof}
  First suppose that (1) holds and let $j_1,j_2,\ldots$ denote the sequence of indices where $\gamma_j(p) > \max(\alpha_j(p),\alpha'_j(p))$. If we set
  \[
    m_r=\sum_{k=1}^r (p-\gamma_{j_k}(p))p^{j_k}
  \]
  and let $A_n$ denote the $n$th coefficient of $F_D(a,b;c;z)$, then we find by Corollary \ref{c:FGM} that $\nu_p(A_{m_r}) \leq -r$. Thus (1) implies (2).

  Conversely, if (2) holds, then by Corollary \ref{c:FGM} there are infinitely many terms where the $j$th digit of $c-1$ is larger than the $j$th digits of both $a\pm b\sqrt{D}-1$.
\end{proof}

Before establishing our density results we prove a simple lemma:
\begin{lem}
\label{l:digitsum}
Let $x,y\in\ZZ_p$ and set $z = x+y$. Then
\[
x[j] + y[j] \geq z[j]-1.
\]
\end{lem}
\begin{proof}
Since $p$-adic addition involves carries, we deduce that
\[
x[j]+y[j]= \begin{cases}
    z[j]& \textrm{no carries at digits $j-1$ and $j$,}\\
    z[j]-1 & \textrm{carry at $j-1$, not at $j$,}\\
    z[j] + \alpha p& \textrm{carry at $j$, not at $j-1$,}\\
    z[j]+\alpha p-1& \textrm{carries at $j-1$ and $j$,}
\end{cases} 
\]
where $\alpha \in \{1,2\}$. Therefore the lemma holds for elementary reasons.
\end{proof}

The density results differ depending on whether $2a$ is an integer or not. We first treat the case where $2a$ is an integer:
\begin{prop}
\label{prop1}
Let $a$ be rational number with $2a \in \ZZ$, let $b$ be rational and nonzero, and let \( c\) be rational and different from a negative integer. Let $K = \QQ(\sqrt{D})$ with $D \in \ZZ$ not a perfect square. Let $M$ be the least common multiple of the denominator of \(c\), and twice the discriminant of \(K\). Consider the hypergeometric series defined by:
\[
F_D= {}_2F_1\left(a+b\sqrt D, a-b\sqrt D; c ; z\right) 
\]
Then there exists a density $\delta$ of $p$-adically bounded primes for $F_D$. If we set
\[
B_K(a;c) \df \{u\in \left(\ZZ/M\ZZ\right)^\times {\vert}
 \{-u^jc\} \leq \{-u^j\tfrac 12\} \text{ for all j and \(u\) splits in K}\},
\]
then $\delta$ satisfies
\[
\frac{\abs{B_K(a;c)}}{\phi(M)} \leq \delta \leq \frac 12.
\]
Furthermore, if Conjecture \ref{conj} holds, then $\delta = \frac{\abs{B_K(a;c)}}{\phi(M)}$. In particular, the density is conjecturally independent of $b$.
\end{prop}
\begin{proof}
In order to apply Lemma \ref{l:digitformula} below, we need to adjust our rational numbers $a$ and $c$ so that they lie in the range $0 < a,c < 1$. This can be done using classical transformation formulas for ${}_2F_1$ relating series whose parameters differ by integers. These transformation formulas will change denominators of the series at finitely many primes, so that it it harmless to assume that $0 < a,c <1$. Henceforth we assume this throughout the proof. In particular, we now take $a = 1/2$.

Since $M$ is even, any integer \(u\) coprime to \(M\) must be odd, the inequality in the definition of $B_K$ simplifies to
\[
\{-u^j c\} \le \frac 12.
\]
By Lemma \ref{l:digitformula}, this implies that if $p$ is congruent to an element of $B_K(a,c)$ mod $M$, then we have
\begin{equation}
\label{eq:cdigitsbound}
(c-1)[j] \le \tfrac{p-1}2    
\end{equation}
for all $j$, where now $x[j]$, for a $p$-adic integer $x$, denotes the $j$th $p$-adic digit of $x$.

Apply Lemma \ref{l:digitsum} with $x = \tfrac 12 + b\sqrt{D}-1$ and $y=\tfrac 12 - b\sqrt{D}-1$. We deduce that
\[
(\tfrac 12 + b\sqrt{D}-1)[j]+(\tfrac 12 - b\sqrt{D}-1)[j] \geq (-1)[j]-1 = p-2.
\]
In particular, at least one of $(\tfrac 12 \pm b\sqrt{D}-1)[j]$ is $\geq (p-1)/2$. But then it follows by Corollary \ref{c:FGM} and equation \eqref{eq:cdigitsbound}, as in the arguments of \cite{FGM}, that carries arising from the denominator parameter $c$ are balanced by the carries arising from the numerator parameters $a\pm b\sqrt{D}$. Therefore, $F_D$ is $p$-adically bounded for all primes congruent to elements of $B_D(a;c)$.

Now assume Conjecture \ref{conj}. For every prime $p$ large enough and such that $p$ is not congruent to an element of $B_K(a;c)$ we will show that $F_D$ is $p$-adically unbounded. If $p$ is such a prime, where the size bound will be identified below, then by definition of $B_K(a;c)$, there exists an integer $j$ such that
\[
\{-p^jc\} > \frac 12.
\]
If $A$ is the multiplicative order of $p \pmod{M}$, we then obtain 
\[
\{-p^{j+nA}c\} > \frac 12
\]
for all integers $n$, since the left side is independent of $n$. Moreover, the left hand side of this inequality is an element of $(1/M)\ZZ$, and so we deduce that:
\[
\{-p^{j+nA}c\} \geq \frac 12 + \frac{1}{M}
\]
for all integers $n$.

Now, notice by Lemma \ref{l:digitformula} that $c-1$ has a periodic $p$-adic expansion of period dividing $A$, since the denominator of $c$ divides $M$. In particular, the terms 
\[(c-1)[j+nA]\] 
are independent of $n$. Also, if $p$ is large enough, Lemma \ref{l:digitformula} and the inequality above imply that 
\[
(c-1)[j+nA] = \floor{\{-p^{j+nA}c\}p} \geq \floor{\frac p2+\frac pM} \geq \frac{p-1}{2}+\floor{\frac pM}
\]
That is, for all $n\geq 0$ we have shown that
\begin{equation}
    \label{eq:cboundprop1}
    \floor{\frac pM} \leq (c-1)[j+nA] - \frac{p-1}{2}.
\end{equation}

Now, Conjecture \ref{conj} yields an infinite subsequence $(x_n)$  of the arithmetic progression $j+nA$ such that
\begin{equation}
    \label{eq:cboundprop1,2}
    \frac{p-1}{2}-\frac 12 \floor{\frac pM} < (\tfrac 12 +b\sqrt{D}-1)[x_n] < \frac{p-1}{2}+\frac 12 \floor{\frac pM}
\end{equation}
for all $n\geq 0$ as long as $p$ is large enough.

Notice that
\[
(\tfrac 12 -b\sqrt{D}-1)[m] = p-1-(\tfrac 12 +b\sqrt{D}-1)[m]
\]
for all $m\geq 0$, since 
\[
(\tfrac 12 -b\sqrt{D}-1) + (\tfrac 12 +b\sqrt{D}-1) = -1 = \sum_{n\geq 0} (p-1)p^n.
\]
If we combine this observation with equation \eqref{eq:cboundprop1,2} we deduce likewise:
\begin{equation}
    \label{eq:negativebound}
    \frac{p-1}{2}-\frac 12 \floor{\frac pM} < (\tfrac 12 -b\sqrt{D}-1)[x_n] < \frac{p-1}{2}+\frac 12 \floor{\frac pM}
\end{equation}
Now equations \eqref{eq:cboundprop1}, \eqref{eq:cboundprop1,2} and \eqref{eq:negativebound} yield
\[
\max\left((\tfrac12+b\sqrt D-1)[x_n],(\tfrac12-b\sqrt D-1)[x_n]
\right) < (c-1)[x_n]
\]
for all $n\geq 0$. Using this, it is now straightforward to construct a sequence of indices along which the coefficients of $F_D$ go to infinity in the $p$-adic absolute value, similarly to the argument in the proof of Theorem \ref{t:rational}. This concludes the proof.
\end{proof}

\begin{cor}
Assume Conjecture \ref{conj}. Then for a set of parameters with \(2a\in\ZZ \), the series $F_D$ has a Dirichlet density of bounded primes equal to $0$ if and only if \(c<\tfrac12\).
\end{cor}
\begin{proof}
As the set of bounded primes is a union of cyclic subgroups of $(\ZZ/M\ZZ)^\times$, the set of bounded primes is empty if and only if \(1\) is not in the bounded prime set. Thus, the Corollary follows from Proposition \ref{prop1}.
\end{proof}

Now we treat the case where $2a \not \in \ZZ$.
\begin{prop}
\label{prop2}
Given two rational parameters $a$ and $c$ where $ a \neq \frac{1}{2}$, and a quadratic number field \(K\) defined by $K = \QQ(\sqrt{D})$, taking M to be the least common multiple of the denominators  of \(a\), \(c\) and the discriminant of \(K\), consider the hypergeometric series defined by:

\[
_2F_1(a+\sqrt D, a-\sqrt D,c ; z) := \sum_{n=0}^\infty \frac{(a+\sqrt D)_n(a-\sqrt D)_n}{(c)_n n!}z^n.
\]

The following set provides a lower bound on the density of bounded primes in the denominators of the series:
\[
B_K(a;c) \df \{u\in \left(\ZZ/M\ZZ\right)^\times {\vert}
\{-u^jc\}\leq\tfrac 12\{-2u^ja\}
 \text{ for all j, and \(u\) splits in K}\}.
\]
This set is \emph{equal} to the set of all equivalence classes such that $F_D$ is $p$-adically bounded for primes in the corresponding congruence if Conjecture \ref{conj} holds.
\end{prop}
\begin{proof}
Let $p$ be a prime that is large enough, so that $p$ is congruent mod $M$ to an element of $B_K(a;c)$. We will show that our hypergeometric series is $p$-adically bounded for all such primes. Thus, we are assuming that $p$ satisfies
\[
\{-p^jc\}\leq \tfrac 12\{-2p^ja\}
\]
for all $j\in \ZZ$. Multiply by $p$, take floors, and use both $\lfloor \tfrac 12 x\rfloor \leq \tfrac 12 \lfloor x\rfloor$ and Lemma \ref{l:digitformula}:
\[
(c-1)[j] \leq \lfloor \tfrac 12\{-2p^ja\}p\rfloor \leq \tfrac 12 \lfloor\{-2p^ja\}p\rfloor \leq\tfrac 12((2a-1)[j]).
\]

That is, we are assuming that for all $j$ we have
\begin{equation}
\label{eq:firstdirection}
(c-1)[j] \leq \tfrac 12 ((2a-1)[j]).    
\end{equation}

Now, by Lemma \ref{l:digitsum} with $x=a+b\sqrt{D}-1$ and $y=a-b\sqrt{D}-1$, so that $z = x+y = 2a-2$, we deduce that
\[
(a+b\sqrt{D}-1)[j]+(a-b\sqrt{D}-1)[j] \geq (2a-2)[j]-1
\]
for all $j$. It follows that for all $j$,
\[
\max ((a+b\sqrt{D}-1)[j],(a-b\sqrt{D}-1)[j]) \geq \tfrac 12((2a-2)[j]-1)
\]
For all but finitely many $j$ we have $(2a-2)[j] = (2a-1)[j]$, and so for all but finitely many $j$ we have
\[
\max ((a+b\sqrt{D}-1)[j],(a-b\sqrt{D}-1)[j]) \geq \tfrac 12((2a-1)[j]-1).
\]
Combining this with inequality \eqref{eq:firstdirection}, we obtain that for all but finitely many $j$ we have
\[
    \max ((a+b\sqrt{D}-1)[j],(a-b\sqrt{D}-1)[j]) \geq (c-1)[j]-\tfrac 12.
\]
But since the left hand side is an integer, as is $(c-1)[j]$, it follows that there exists an $N$ such that for all $j > N$ we have
\begin{equation}
\label{eq:maxwithc}
    \max ((a+b\sqrt{D}-1)[j],(a-b\sqrt{D}-1)[j]) \geq (c-1)[j].
\end{equation}
Now again, the proof can be completed in this case using Corollary \ref{c:FGM} and the argument of \cite{FGM}: equation \eqref{eq:maxwithc} assures that carries arising in the formula of Corollary \ref{c:FGM} from $c-1$ are balanced by carries arising from at least one of the numerator parameters $a\pm b\sqrt{D}-1$. Hence, $F_D$ is $p$-adically integral for all such primes congruence to elements of $B_K(a;c)$.

Now, assume Conjecture \ref{conj}, and suppose that $p$ is a prime that is not congruent to an element of $B_K(a;c)$. This means that for some $j$ we have 
\[
\{-p^jc\} > \tfrac12\{-2p^ja\}.
\]
In the inequality above, both sides are rational numbers such that multiplying by $M$ yields an integer. Thus, the inequality above implies that \(\{-p^jc\} \ge \tfrac12\{-2p^ja\} + \frac1M\). Then we deduce
\[
\floor{\{-p^jc\}p} \ge \floor{\tfrac12\{-2p^ja\}p + \frac{p}{M}} \geq \floor{\tfrac12\{-2p^ja\}p} + \floor{\frac{p}{M}}.
\]
Therefore, by Lemma \ref{l:digitformula} and the above mentioned inequality \(\floor{\tfrac12x}\geq\tfrac12\floor{x}\), we arrive at
\begin{equation}
\label{e:otherway}
(c-1)[j] \geq \tfrac12((2a-1)[j])+\floor{\tfrac pM}.
\end{equation}
And as above, if $A$ is the order of $p \pmod{M}$, then the expansion of $c-1$ is periodic of period $A$, and the expansion of $2a-1$ is eventually periodic of period $A$. So, after possibly replacing $j$ by some large term $j+nA$ to avoid the part of the expansion of $2a-1$ that is not periodic, we deduce that
\begin{equation}
    \label{eq:cboundprop2}
    \floor{\tfrac pM}+\tfrac12((2a-1)[j+nA]) \leq (c-1)[j+nA]  
\end{equation}
for all integers $n\geq 0$, since both sequences of digits in inequality \eqref{eq:cboundprop2} are independent of $n$.

Up until this point, the proof has been essentially the same as the proof of Proposition \ref{prop1}. At this point the proofs diverge, as now the relationship between the digits of $a+b\sqrt{D}-1$ and $a-b\sqrt{D}-1$ is not as simple as in the case when $a=\tfrac 12$.

Let us first observe that if $p$ is large enough, where this restriction on the size only depends on $a$, then all of the $p$-adic digits of $2a-1$ in its periodic part are nonzero by the last claim in Lemma \ref{l:digitformula}. Therefore, we may suppose that $(2a-1)[j] \neq 0$. Choose a constant $K > 1$ such that
\[
\frac{1}{KM} < \frac{1}{2p}((2a-1)[x_n]).
\]
It may look as though this constant depends on $p$, but by Lemma \ref{l:digitformula}, when $p$ is large, the digits of $(2a-1)$ are all approximately equal to rational numbers of the form $\alpha p/M$ for $\alpha =1,2,\ldots, M-1$. Hence when $p$ is large enough, the constant $K$ can be chosen independently of $p$.

Conjecture \ref{conj} yields an infinite subsequence $(x_n)$ of the arithmetic progression \((j+nA)\) such that
\begin{equation}
    \label{eq:bsqrtDprop2}
    \tfrac12((2a-1)[x_n])-\tfrac 1K\floor{\tfrac pM} < (a+b\sqrt D-1)[x_n] < \tfrac12((2a-1)[x_n])+\tfrac 1K\floor{\tfrac pM} 
\end{equation}
for all $n\geq 0$, assuming $p$ is large enough.

As in the proof of Lemma 4.7, we have one of the following four possibilities depending on how certain carries take place when adding $a\pm b\sqrt{D}-1$:
\[
(a+b\sqrt{D}-1)[x_n] + (a-b\sqrt{D}-1)[x_n] =\begin{cases}
(2a-1)[x_n] + \alpha p,\\
(2a-1)[x_n]-1 + \alpha p,\\
\end{cases}
\]
where $\alpha \in \{0,1,2\}$, for $j$ large enough so that $(2a-1)[x_n]=(2a-2)[x_n]$. However, from equation \eqref{eq:bsqrtDprop2} and the trivial bound $(a-b\sqrt{D}-1)[x_n] \leq p-1$ we obtain
\[
(a+b\sqrt{D}-1)[x_n] + (a-b\sqrt{D}-1)[x_n]\leq \tfrac12((2a-1)[x_n])+\tfrac 1K\floor{\tfrac pM} + p-1.
\]
It follows that
\[
\alpha  \leq \tfrac {1}{pK}\floor{\tfrac pM} + 1-\tfrac{1}{2p}((2a-1)[x_n]) \leq 1+\tfrac {1}{KM}-\tfrac{1}{2p}((2a-1)[x_n]) < 1,
\]
where the last inequality follows by our choice of $K$. Therefore, $\alpha=0$ and we have
\[
(a+b\sqrt{D}-1)[x_n] + (a-b\sqrt{D}-1)[x_n]\leq (2a-1)[x_n],
\]
so that if we apply the left side of inequality \eqref{eq:bsqrtDprop2} then
\[
(a-b\sqrt{D}-1)[x_n] \leq \tfrac 12((2a-1)[x_n] + \tfrac{1}{K}\floor{\tfrac pM}
\]
Now, the line above and inequalities \eqref{eq:bsqrtDprop2} and \eqref{eq:cboundprop2} give
\[
\max((a+b\sqrt{D}-1)[x_n], (a-b\sqrt{D}-1)[x_n]) < (c-1)[x_n]
\]
for all $n$ large enough, and thus $F_D$ is unbounded at $p$ using the same argument as in Proposition \ref{prop1}
\end{proof}

\begin{cor}
If Conjecture \ref{conj} holds, then a set of parameters $(a,c)$ as above with \(a\not\in \tfrac12 \ZZ\) has series $F_D$ with a density of bounded primes equal to zero if and only if \(2\{-c\}>\{-2a\}\).
\end{cor}
\begin{proof}
Notice that the set of bounded primes is a union of cyclic subgroups, if nonempty. Therefore, the set of bounded primes is empty if and only if $1$ is not in the bounded prime set. This is equivalent to the statement of the Corollary by the preceding proposition.
\end{proof}

\section{Towards a quadratic irrational Schwarz list}
\label{s:schwarz}

Obviously, due to the splitting condition on $p$ in $K = \QQ(\sqrt{D})$, the density
\[
D_K(a;c) = \frac{\abs{B_K(a;c)}}{\phi(M)}
\]
is at most equal to $1/2$. In analogy with the Schwarz list, it is natural to ask how often this maximal density is reached. Computations suggest that frequently we have $D_K(a;c) \leq 1/4$, but this is not always the case.


To understand why densities frequently do not exceed $1/4$, let us introduce some more notation: let $M_1$ denote the lcm of the denominators of $a$ and $c$, and let $M$ denote the lcm of $M_1$ and $\Delta_K$, where $\Delta_K$ is the discriminant of our field $K=\QQ(\sqrt{D})$. Define $B(a;c)$ in the same way as $B_K(a;c)$ except we use $M_1$ in place of $M$ and we ignore the splitting condition. Then define
\begin{equation}
\label{e:dac}
D(a;c) = \frac{\abs{B(a;c)}}{\phi(M_1)}.
\end{equation}

Suppose that $M_1$ and $\Delta_K$ are coprime. Then by the Chinese remainder theorem we have
\[
(\ZZ/M\ZZ)^\times \cong (\ZZ/M_1\ZZ)^\times \times (\ZZ/\Delta_K\ZZ)^\times
\]
and this isomorphism induces a bijection of sets:
\begin{equation}
\label{eq:coprimeDbound}
B_K(a;c) \cong B(a;c) \times \{u\in (\ZZ/\Delta_K\ZZ)^\times \mid u \textrm{ splits in }K\}.
\end{equation}
\begin{prop}
\label{p:Dbound}
    Suppose that $a,c$ are rationals such that \(2c\) and \(2a\) are not both integers. Then 
    \[
    D(a;c) \leq 1/2.
    \]
\end{prop}
\begin{proof}
We first consider the case where \(2a\not\in\ZZ\). Notice that
\[
B(a;c) \subseteq A(a;c) = \{u \in (\ZZ/M_1\ZZ)^\times \mid \{-uc\} \leq \tfrac 12 \{-2ua\}\}.
\]
It suffices to show that at most one of \(u\) and \(-u\) are in the set \(A(a;c)\). 
Assume towards a contradiction that both $u$ and $-u$ are in A. Considering $u \in A(a;c)$, we have:
\begin{align*}
\{-uc\} &\leq \tfrac{1}{2}\{-2ua\}\\
\intertext{Noting $ \{-x\} = 1 - \{x\}$ holds when $x \in \QQ\setminus\ZZ$,}
1 - \{uc\} & \leq \tfrac{1}{2}(1 - \{2ua\})\\
1 - \{uc\} & \leq \tfrac{1}{2} - \tfrac{1}{2}\{2ua\}\\
- \{uc\} & \leq -\tfrac{1}{2} - \tfrac{1}{2}\{2ua\}\\
\tfrac{1}{2} + \tfrac{1}{2}\{2ua\} & \leq  \{uc\}.
\end{align*}
As $-u \in A(a;c)$, we also have 
\[\{uc\} \leq \tfrac{1}{2}\{2ua\}.\]
Therefore,
\[\tfrac{1}{2} + \tfrac{1}{2}\{2ua\} \leq \{uc\} \leq \tfrac{1}{2}\{2ua\},\] a contradiction! Thus only one of \(u\) and \(-u\) are in \(A(a;c)\) in this case.

In the case where \(2a\in\ZZ\), we must instead consider
\[
B(a;c)\subseteq A(a;c) = \{u\in(\ZZ/M\ZZ)^\times \mid \{-uc\}\le\tfrac12\}.
\]
Once again, it suffices to show that at most half of all \(u\) can appear in this set. 
Assume towards a contradiction that \(u\) and \(-u\) are both in this set. Taking \(u\in A(a;c)\), we have:
\begin{align*}
\{-uc\}     &\le \tfrac12\\
1-\{uc\}    &\le \tfrac12\\
-\{uc\}     &\le -\tfrac12\\
\tfrac12    &\le \{uc\}
\end{align*}

By our assumptions, \(-u\in A(a;c)\), so
\[
\{uc\} \le \tfrac12.
\]

Thus, since we have assumed that in this case \(2c\not\in\ZZ\), this is a contradiction.
So \(B(a;c)\) contains at most half of the elements of \((\ZZ/M_1\ZZ)^\times\), and so \(D(a;c)\) is at most \(\tfrac12\).

\end{proof}

Now we can prove that the field $K$ must be ramified at some of the primes dividing our hypergeometric parameters in order to find examples $(a,c,K)$ where $D_K(a;c)=\tfrac 12$:
\begin{cor}
\label{cor:Dbound}
    Suppose that $\gcd(\Delta_K,M_1) = 1$ and that $2a$ and $2c$ are not both in $\ZZ$. Then
    \[
    D_K(a;c) \leq 1/4.
    \]
\end{cor}
\begin{proof}
    Combining equation \eqref{eq:coprimeDbound} and Proposition \ref{p:Dbound} we find that
    \[
    D_K(a;c) \leq D(a;c)\times \tfrac 12 \leq \tfrac 14.
    \]
\end{proof}

Table \ref{tab:largedensitiesnoK} lists pairs of rational parameters $a,c \in (0,1)$ with $a$ and $c$ of height at most $48$, and where $2a \not \in \ZZ$, such that $D(a,c) > 1/4$. We do not know if there exists an infinite number of such examples as the height of $a$ and $c$ grow.

Since we have $D_K(a,c) \leq D(a,c)$ for all fields $K = \QQ(\sqrt{D})$, in each case from Table \ref{tab:largedensitiesnoK}, it is natural to ask whether we can find a field where equality holds. By Corollary \ref{cor:Dbound}, we know that in such examples, $D$ can't be coprime to the least common multiple of the denominators of $a$ and $c$. We thus examined all choices of $K=\QQ(\sqrt{D})$ for $D$ dividing this least common multiple. In all cases of pairs $(a,c)$ from Table \ref{tab:largedensitiesnoK} with $D(a,c)=1/2$, we found exactly one field $K$ such that $D_K(a,c)=1/2$. Further, it was always imaginary quadratic.

\nocite{MNT}
\nocite{kummer}

\begin{table}
\resizebox{\textwidth}{!}{
\renewcommand{\arraystretch}{1}
    \centering
    \begin{tabular}{p{2.8cm}p{2.8cm}p{2.8cm}p{2.8cm}p{2.8cm}p{2.8cm}}
    $D(a;c)=\tfrac12$  &$D(a;c)=\tfrac12$  & $D(a;c) = \tfrac38$& $D(a;c) = \tfrac38$&$D(a;c)=\tfrac5{16}$&$D(a;c)=\tfrac13$\\
    \begin{tabular}{cc}
         $a$&$c$ \\
         \hline
         1/3&5/6\\
         2/3&2/3\\
         2/3&5/6\\
         1/4&3/4\\
         1/4&5/6\\
         3/4&3/4\\
         3/4&5/6\\
         1/6&2/3\\
         1/6&5/6\\
         5/6&5/6\\
         1/8&3/4\\
         1/8&5/8\\
         3/8&7/8\\
         5/8&3/4\\
         5/8&5/8\\
         7/8&7/8\\
         1/12&2/3\\
         1/12&5/6\\
         1/12&7/12\\
         5/12&11/12\\
         7/12&2/3\\
         7/12&5/6\\
         7/12&7/12\\
         11/12&11/12\\
         1/16&5/8\\
         3/16&7/8\\
         9/16&5/8\\
         11/16&7/8\\
         1/20&11/20\\
         3/20&13/20\\
         7/20&17/20\\
         9/20&19/20\\
         11/20&11/20\\
         13/20&13/20\\
         17/20&17/20\\
         19/20&19/20
    \end{tabular}
    &
    \begin{tabular}{cc}
    $a$&$c$\\
    \hline
    1/24&3/4\\
    1/24& 5/6\\
    1/24&7/12\\
    1/24&13/24\\
    5/24&3/4\\
    5/24&11/12\\
    5/24&17/24\\
    7/24&5/6\\
    7/24&19/24\\
    11/24&23/24\\
    13/24&3/4\\
    13/24&5/6\\
    13/24&7/12\\
    13/24&13/24\\
    17/24&3/4\\
    17/24&11/12\\
    17/24&17/24\\
    19/24&5/6\\
    19/24&19/24\\
    23/24&23/24\\
    1/40&11/20\\
    3/40&13/20\\
    7/40&17/20\\
    9/40&19/20\\
    21/40&11/20\\
    23/40&13/20\\
    27/40&17/20\\
    29/40&19/20\\
    1/48&13/24\\
    5/48&17/24\\
    7/48&19/24\\
    11/48&23/24\\
    25/48&13/24\\
    29/48&17/24\\
    31/48&19/24\\
    35/48&23/24
    \end{tabular}
    &
    \vspace{-9.215cm}
    \begin{tabular}{cc}
     $a$&$c$ \\
     \hline
     1/16&5/6\\
     1/16&7/12\\
     1/16&11/12\\
     1/16&17/24\\
     1/16&19/24\\
     3/16&5/6\\
     3/16&17/24\\
     3/16&19/24\\
     5/16&5/6\\
     5/16&11/12\\
     5/16&23/24\\
     7/16&23/24\\
     9/16&5/6\\
     9/16&7/12\\
     9/16&11/12\\
     9/16&17/24\\
     9/16&19/24\\
     11/16&5/6\\
     11/16&17/24\\
     11/16&19/24\\
     13/16&5/6\\
     13/16&11/12\\
     13/16&23/24\\
     15/16&23/24
     \end{tabular}
    &
    \vspace{-9.215cm}
    \begin{tabular}{cc}
     $a$&$c$ \\
     \hline
     1/48&5/6\\
     1/48&7/12\\
     1/48&17/24\\
     1/48&23/24\\
     5/48&11/12\\
     5/48&19/24\\
     7/48&5/6\\
     7/48&17/24\\
     7/48&23/24\\
     11/48&19/24\\
     13/48&5/6\\
     17/48&11/12\\
     25/48&7/12\\
     25/48&17/24\\
     25/48&5/6\\
     25/48&23/24\\
     29/48&11/12\\
     29/48&19/24\\
     31/48&5/6\\
     31/48&17/24\\
     31/48&23/24\\
     35/48&19/24\\
     37/48&5/6\\
     41/48&11/12
    \end{tabular}&
    \vspace{-9.215cm}
    \begin{tabular}{cc}
    $a$&$c$\\
    \hline
    1/32&11/12\\
    5/32&23/24\\
    9/32&11/12\\
    13/32&23/24\\
    17/32&11/12\\
    21/32&23/24\\
    25/32&11/12\\
    29/32&23/24\\
    1/40&5/6\\
    3/40&5/6\\
    7/40&5/6\\
    9/40&5/6\\
    21/40&5/6\\
    23/40&5/6\\
    27/40&5/6\\
    29/40&5/6
    \end{tabular}
    &
    \vspace{-9.215cm}
    \begin{tabular}{cc}
         $a$&$c$ \\
         \hline
         2/7&5/6\\
         11/14&5/6\\
         1/21&5/6\\
         2/39&5/6\\
         11/39&5/6\\
         23/42&5/6
    \end{tabular}
    \end{tabular}
    }
    \caption{All examples where $D(a;c) >1/4$ for $0<a,c<1$ with $2a\not \in \ZZ$, and the height of $a$ and $c$ are at most $48$.}
    \label{tab:largedensitiesnoK}
\end{table}

\begin{table}
\renewcommand{\arraystretch}{1.3}
    \centering
    \begin{tabular}{ccc}
         $a$&$c$&$D$\\
         \hline
         1/3&5/6&-3\\
         2/3&2/3&-3\\
         2/3&5/6&-3\\
         1/4&3/4&-1\\
         1/4&5/6&-3\\
         3/4&3/4&-1\\
         3/4&5/6&-3\\
         1/6&2/3&-3\\
         1/6&5/6&-3\\
         5/6&5/6&-3\\
         1/8&3/4&-1\\
         1/8&5/8&-2\\
         3/8&7/8&-2\\
         5/8&3/4&-1\\
         5/8&5/8&-2\\
         7/8&7/8&-2\\
         1/12&2/3&-3\\
         1/12&5/6&-3\\
         1/12&7/12&-1\\
         5/12&11/12&-1\\
         7/12&2/3&-3\\
         7/12&5/6&-3\\
         7/12&7/12&-1\\
         11/12&11/12&-1\\
         1/16&5/8&-2\\
         3/16&7/8&-2\\
         9/16&5/8&-2\\
         11/16&7/8&-2\\
         1/20&11/20&-5\\
         3/20&13/20&-5\\
         7/20&17/20&-5\\
         9/20&19/20&-5\\
         11/20&11/20&-5\\
         13/20&13/20&-5\\
         17/20&17/20&-5\\
         19/20&19/20&-5
    \end{tabular}
    \hspace{1cm}
    \begin{tabular}{ccc}
    $a$&$c$&$D$\\
    \hline
    1/24&3/4&-1\\
    1/24& 5/6&-3\\
    1/24&7/12&-1\\
    1/24&13/24&-6\\
    5/24&3/4&-1\\
    5/24&11/12&-1\\
    5/24&17/24&-6\\
    7/24&5/6&-3\\
    7/24&19/24&-6\\
    11/24&23/24&-6\\
    13/24&3/4&-1\\
    13/24&5/6&-3\\
    13/24&7/12&-1\\
    13/24&13/24&-6\\
    17/24&3/4&-1\\
    17/24&11/12&-1\\
    17/24&17/24&-6\\
    19/24&5/6&-3\\
    19/24&19/24&-6\\
    23/24&23/24&-6\\
    1/40&11/20&-5\\
    3/40&13/20&-5\\
    7/40&17/20&-5\\
    9/40&19/20&-5\\
    21/40&11/20&-5\\
    23/40&13/20&-5\\
    27/40&17/20&-5\\
    29/40&19/20&-5\\
    1/48&13/24&-6\\
    5/48&17/24&-6\\
    7/48&19/24&-6\\
    11/48&23/24&-6\\
    25/48&13/24&-6\\
    29/48&17/24&-6\\
    31/48&19/24&-6\\
    35/48&23/24&-6
    \end{tabular}
    \caption{Examples of \(a,c,D\) such that $D_K(a;c) = 1/2$ for \(K = \QQ(D)\).}
    \label{tab:dvaryk}
\end{table}

\bibliographystyle{plain}
\bibliography{refs}
\end{document}